\documentclass[reqno,12pt]{amsart}
\author{Julia Brandes}
\title[Systems of quadratic and cubic diagonal equations]{The Hasse principle for systems of quadratic and cubic diagonal equations}
\address{Mathematical Sciences, Chalmers Institute of Technology and University of Gothenburg, 412 96 G{\"o}teborg, Sweden}
\email{brjulia@chalmers.se}


\usepackage[english,ngerman]{babel}
\addto\extrasenglish{\languageshorthands{ngerman}}
\usepackage[T1]{fontenc}
\usepackage[latin1]{inputenc}
\usepackage{latexsym}
\usepackage[T1]{fontenc}
\usepackage[latin1]{inputenc}
\usepackage{latexsym}
\usepackage[arrow, matrix, curve]{xy}
\usepackage[dvips]{graphicx}
\usepackage{epsfig}
\usepackage{bm}
\usepackage{amssymb}
\usepackage{amsmath}
\usepackage{verbatim}
\usepackage{amsthm}

\usepackage{xspace}

\usepackage{enumerate}
\usepackage{mathrsfs}
\usepackage[hmargin=3cm,vmargin=3cm]{geometry}
\usepackage{scalerel}
\usepackage{float}
\usepackage{subcaption}

\usepackage{tikz}
\usetikzlibrary{calc}
\usepackage[babel=true,kerning=true]{microtype}
\usetikzlibrary{patterns}

\def\Z{\mathbb Z}
\def\Q{\mathbb Q}
\def\R{\mathbb R}

\def\T{\mathbb T}

\def\B#1{\mathbf{#1}}
\def\ba{\bm{\alpha}}
\def\bb{\bm{\beta}}
\def\bg{\bm{\gamma}}
\def\bd{\bm{\delta}}

\def\F#1{\mathfrak{#1}}
\def\cal#1{\mathcal{#1}}

\def\D{\;\mathrm{d}}

\def\a0{\alpha_0}

\def\eps{\varepsilon}

\def\sm#1{\scaleobj{.5}{#1}}
\def\ssq{\scaleobj{.5}{\square}}

\DeclareMathOperator{\Id}{Id}

\DeclareMathOperator{\diag}{diag}

\renewcommand{\le}{\leqslant}
\renewcommand{\ge}{\geqslant}

\newtheorem{thm}{Theorem}[section]
\newtheorem{lem}[thm]{Lemma}

\newtheorem{prop}[thm]{Proposition}

\theoremstyle{definition}

\theoremstyle{remark}

\numberwithin{equation}{section}

\newenvironment{pf}{\begin{proof}[Proof]}{\end{proof}}

\begin{document}
\selectlanguage{english}

\begin{abstract}
Employing Br\"udern's and Wooley's new complification method, we establish an asymptotic Hasse principle for the number of solutions to a system of $r_3$ cubic and $r_2$ quadratic diagonal forms, where $r_3 \ge 2r_2>0$, in $s \ge 6r_3+\lfloor (14/3)r_2\rfloor + 1$ variables.
\end{abstract}
\maketitle

\section{Introduction}
In this memoir we are concerned with systems of diophantine equations of the shape
\begin{align}\label{sys}
  \sum_{i=1}^s c_{j,i}^{(2)} x_i^2 = \sum_{i=1}^s c_{h,i}^{(3)}x_i^3 =0 \qquad (1 \le j \le r_2, 1 \le h \le r_3),
\end{align}
where $c_{i,j}^{(k)}$ are integers.
It is commonly acknowledged that, unless fundamentally new ideas become available that avoid the implicit use of mean values, at least $6r_3 + 4r_2 + 1$ variables are required in order to establish asymptotic estimates for the number of integral solutions of the system \eqref{sys}. This theoretical limit has recently been attained by Wooley \cite[Theorem 1.1]{W:15sae5} in the case $r_2 = r_3 = 1$, and by the author jointly with Parsell \cite[Theorem 1.4]{BP:16} for systems consisting of $r_2 \ge 1$ quadratic forms and one cubic equation. The latter work applies a disentangling argument going back in its essence to the methods of Davenport and Lewis \cite{DL:69}, and provides estimates for the number of variables required to establish a Hasse principle and asymptotic formul{\ae} for the number of solutions of general systems of additive equations involving different degrees with arbitrary multiplicities.
In the case of purely cubic systems, these classical methods hit a boundary when it comes to establishing solubility of systems of $r_3$ equations in fewer than roughly $7r_3$ variables (see e.g. Br\"udern and Cook \cite{BC:92}). However, new ideas have recently become available in the work of Br\"udern and Wooley \cite{BW:16corr, BW:16h3} that achieve essentially square root cancellation in this case. This opens up the possibility that a modification of their methods may lead to stronger bounds for mixed systems of cubic and quadratic equations also. The objective of this paper is to carry out these modifications and establish asymptotic formul{\ae} for the number of solutions to such mixed systems consisting of at least twice as many cubic as quadratic equations, using fewer variables than hitherto necessary.

For a large integer $P$ let $N(P)$ denote the number of integral vectors $\B x \in [-P,P]^s$ satisfying \eqref{sys}. It is clear that a non-singularity condition of some sort is required to ensure that the equations in \eqref{sys} do not interact in any non-generic way. We say that an $r \times s$ matrix $A$ is \emph{highly non-singular} if any collection of $r$ columns of $A$ forms a non-singular submatrix. In this notation our result is as follows.
\begin{thm}\label{thm}
  Let $r_3 \ge 2 r_2 >0$ and $s \ge  6r_3+\lfloor (14/3)r_2\rfloor + 1$, and suppose that the matrices $C^{(2)} = (c_{j,i}^{(2)})$ and $C^{(3)}=(c_{h,i}^{(3)})$ are highly non-singular. Then one can find a parameter $\delta>0$ such that
  \begin{align*}
      N(P) = (c + O(P^{-\delta})) P^{s-3r_3-2r_2},
  \end{align*}
  where $c$ is a non-negative constant encoding the density of real and $p$-adic solutions to the system \eqref{sys}.
\end{thm}
For comparison, Theorem 1.4 of \cite{BP:16} establishes the same conclusion under the more stringent hypothesis that $s \ge 8r_3+\lfloor(8/3)r_2\rfloor +1$, and proves a Hasse principle without asymptotic formula for $s \ge 7r_3 + \lceil (11/3)r_2 \rceil$.
Observe in particular that in the case $r_2=1$ Theorem \ref{thm} yields a bound on the number of variables given by $s \ge 6r_3+5 = 2(3r_3 + 2)+1$, so for systems of one quadratic and $r_3 \ge 2$ cubic equations we attain the theoretical limit imposed by square root cancellation.

Two points deserve further remarks. Firstly, one notes the slightly irritating hypothesis $r_3 \ge 2 r_2$. This is a technical condition arising from the idiosyncrasies of the method, which it does not seem easy to circumvent.
Secondly, we are making no statement here as to whether or not the constant $c$ is actually positive. In \S \ref{CM} we will see that $c$ can be written as a product of the solution densities of \eqref{sys} over the completions of $\Q$, where each individual factor is positive if the system \eqref{sys} has a non-singular solution in the respective local field. Unfortunately, the conditions required to guarantee local solubility are typically much more stringent than what is needed to establish local-global principles. For instance, the work of Knapp \cite{Knapp:09} shows that systems of the type \eqref{sys} have non-trivial $p$-adic solutions for all odd primes $p$ whenever $s>(75/2)(r_2+r_3)^3$. Whilst it would be desirable to establish bounds of the quality $s > 4r_2+9r_3$ as  conjectured by Artin, in the light of Wooley's work \cite{W:15ac} it is not clear whether such a result is even feasible to aim for. Nonetheless, even the bound hypothesised by Artin would require more variables for $p$-adic solubility than we need for a local-global principle.\\


The following notational conventions will be observed throughout the paper. Any expression involving the letter $\eps$ will be true for any (sufficiently small) $\eps > 0$. Consequently, no effort will be made to track the respective `values'  of $\eps$. Also, any statement involving vectors is to be understood componentwise. In this spirit, we write $(q,\B b)= \gcd (q,b_1, \ldots, b_n)$ whenever $\B b \in \Z^n$, and we interpret a vector inequality of the shape $C \le \B b \le D$ to mean that $C \le b_i \le D$ for $i=1, \ldots, n$. Write $\Id_k$ for the $k \times k$ identity matrix, and set $\T = \R/\Z$. Finally, all implied constants may depend on $s$, $r_2$ and $r_3$ as well as the coefficient matrices $C^{(3)}$ and $C^{(2)}$, but are independent of $P$, which we take to be a large integer.\\

The author is very grateful to the referee, whose comments led to a greatly improved paper.

\section{Totally non-singular matrices and auxiliary mean values}

For the proof of Theorem \ref{thm} we make use of auxiliary matrices similar to those introduced by Br\"udern and Wooley \cite{BW:16corr}. We call an $r \times s$ matrix $A$ \emph{highly non-singular} if $s \ge r$ and every collection of $r$ columns of $A$ is linearly independent, and \emph{totally non-singular} if $A$ has no vanishing minors.
\begin{lem}\label{hns-prop}
    We have the following properties of highly non-singular matrices.
    \begin{enumerate}[(i)]
        \item
            \begin{enumerate}[(a)]
                \item Any matrix obtained from a highly non-singular matrix via elementary row operations is also highly non-singular.
                \item Suppose that $A$ is a highly non-singular $r \times s$ matrix with $s \ge r+1$, then the matrix obtained from $A$ by deleting an arbitrary column is also highly non-singular.
                \item If $A$ is highly non-singular with the property that one of $A$'s columns contains only one non-zero element, then the matrix obtained from $A$ by removing that column and the row containing the non-zero element is also highly non-singular.
            \end{enumerate}
        \item
            An $r \times s$ matrix $B$ is totally non-singular if and only if the $r \times (r+s)$ matrix $(\Id_r | B)$ is highly non-singular.
    \end{enumerate}
\end{lem}
\begin{proof}
    The first three statements are immediate from the definition of high non-singularity (see also Lemma~2.1 in \cite{BW:16h3}), and the latter statement is a trivial generalisation of Lemma~3.1 of \cite{BW:16corr}.
\end{proof}

For natural numbers $n \ge 2$, $l$ and  $i_1 < \ldots < i_n$ and $j_1 < \ldots < j_n$ we call $D$ a \emph{linked-block matrix of type $(n,l)$} if $D$ is obtained from an $i_1 \times j_1$ matrix $A_1$, and $(i_m-i_{m-1}+l) \times (j_m-j_{m-1})$ matrices $B_m$ with their lower right corner at $(i_m, j_m)$ for $2 \le m \le n$. The matrix $D$ should be thought of as having been obtained from a conventional block matrix $\diag(A_1, A_2, \ldots, A_n)$ composed of matrices $A_m$ of format $(i_m-i_{m-1}) \times (j_m-j_{m-1})$, where adjacent blocks are connected by $l$ linking rows.

Let $V_k$ be a totally non-singular matrix of format $r \times (r-l)$ when $2 \le k \le n$, and of format $t \times (t-l+\omega)$ for $k=1$. Here and henceforth we will assume that
\begin{align*}
    r \ge 2l, \quad t \ge l, \quad 0 \le \omega  \le l.
\end{align*}
Now use the matrices $V_k$ to form the linked block matrix $V$ of type $(n, l)$. We call $D$ an \emph{auxiliary matrix of type $(n,t,\omega)_{r,l}$} if $D$ is of block shape $D = (U, V)$, where $V$ is a linked-block matrix as above and $U$ is a non-singular diagonal matrix of size $(n-1)(r-l)+t$. Then $D$ is a matrix of format $R \times S$ with $R=(n-1)(r-l)+t$ and $S = 2R-l+\omega$. For instance, the reader may check that the matrix
\begin{align*}
    \left(\begin{smallmatrix}1&9&1 \\ 2&7&7\\ 8&4&3\\ 3&1&7\\3&7&9
    \end{smallmatrix}\right)
\end{align*}
is totally non-singular, and thus the matrix
\begin{align*}
    \left(\begin{smallmatrix}
        1& & & & & & & & & & &1&9&1& & & & & & \\
         &1& & & & & & & & & &2&7&7& & & & & & \\
         & &1& & & & & & & & &8&4&3& & & & & & \\
         & & &1& & & & & & & &3&1&7&1&9&1& & & \\
         & & & &1& & & & & & &3&7&9&2&7&7& & & \\
         & & & & &1& & & & & & & & &8&4&3& & & \\
         & & & & & &1& & & & & & & &3&1&7&1&9&1\\
         & & & & & & &1& & & & & & &3&7&9&2&7&7\\
         & & & & & & & &1& & & & & & & & &8&4&3\\
         & & & & & & & & &1& & & & & & & &3&1&7\\
         & & & & & & & & & &1& & & & & & &3&7&9
     \end{smallmatrix} \right)
\end{align*}
is an auxiliary matrix of type $(3,5,0)_{5,2}$. Here we followed the convention that zero entries be omitted. Were one to delete the first one or two rows and columns, one would end up with an auxiliary matrix of type $(3,4,1)_{5,2}$ or $(3,3,2)_{5,2}$, respectively.
Note that our definition of an auxiliary matrix has been simplified compared to that of Br\"udern and Wooley \cite{BW:16corr} so as to make the following arguments somewhat slicker; it turns out that this can be done without adding any significant complications later in the argument. We also remark that the definition of auxiliary matrices can be extended in the natural way to the case $n=1$.

Let now $D$ be an auxiliary matrix of type $(n,t,\omega)_{r,l}$. We define the cubic exponential sum
\begin{align*}
 g(\eta) = \sum_{x \in [-P, P]} e(\eta x^3)
\end{align*}
and make the change of variables
\begin{align}\label{chvar}
 \theta_j = \sum_{i=1}^R d_{i,j} \eta_i.
\end{align}
Our first goal is a bound for the mean value
\begin{align*}
 I(P,D) = \oint \prod_{i=1}^R |g(\theta_{i})|^2 \prod_{i=R+1}^S |g(\theta_{i})|^4 \D {\bm \eta},
\end{align*}
where we introduced the shorthand notation $\oint$ for the integral over the (in this case) $R$-dimensional unit cube.
For convenience, we will write $I_{n,t}^{\omega}(P) = \sup I(P,D)$, where the supremum is taken over all auxiliary matrices $D$ of type $(n,t,\omega)_{r,l}$; the respective values of $r$ and $l$ will stay fixed throughout the argument.

\begin{prop}\label{silly}
    Suppose that $r \ge 2l$. For all integral parameters $n \ge 1$, $t \ge l$ and $0 \le \omega  \le l$ we have
    \begin{align*}
        I_{n,t}^{\omega}(P) \ll P^{3((n-1)(r-l)+t +\omega) -2l  + \eps}.
    \end{align*}
\end{prop}

The proof of the proposition is by an inductive argument distinguishing several cases. In the proofs we will repeatedly consider submatrices that arise from deleting a certain set of columns and rows. It will be convenient to denote the submatrix $(d_{i,j})$ of $D$ consisting of rows with indices $a \le i \le b$ and columns $c \le j \le d$ by $[a,b] \times [c,d]$.

\begin{lem}\label{ind1}
    Suppose that $t \ge l+1$ and $n \ge 2$, then we have
    \begin{align*}
        I_{n,t}^{0}(P) \ll  P^{3(t-l)+\eps} I_{n-1,r}^{0}(P) + P^\eps I_{n,t-1}^{1}(P).
    \end{align*}
\end{lem}
\begin{proof}
    Let $D$ be auxiliary of type $(n,t,0)_{r,l}$ and observe that by orthogonality, $I(P,D)$ counts the number of solutions to the system
    \begin{align}\label{ind1-sys}
        d_{i,i}(x_{i,1}^3 - x_{i,2}^3) + \sum_{j=R+1}^S d_{i,j}(x_{j,1}^3+x_{j,2}^3-x_{j,3}^3-x_{j,4}^3) = 0 \qquad (1 \le i \le R)
    \end{align}
    with $-P \le x_{j,k} \le P$ for all $j$ and $k$. Denote by $T_0$ the number of solutions counted by \eqref{ind1-sys} having $x_{j,1}=x_{j,2}$ for all $1 \le j \le t-l$, and write $T_j$ for the number of solutions having $x_{j,1} \neq x_{j,2}$. Then we have
    \begin{align*}
        I(P, D) \le T_0 + T_1 + \ldots + T_{t-l},
    \end{align*}
    and one sees easily that $T_0 \ll P^{t-l}H_0$, where $H_0$ denotes the number of solutions to the system
    \begin{align*}
        \sum_{j=R+1}^S d_{i,j}(x_{j,1}^3+x_{j,2}^3-x_{j,3}^3-x_{j,4}^3) &= 0 \qquad (1 \le i \le t-l), \\
        d_{i,i}(x_{i,1}^3 - x_{i,2}^3) + \sum_{j=R+1}^S d_{i,j}(x_{j,1}^3+x_{j,2}^3-x_{j,3}^3-x_{j,4}^3) &= 0 \qquad (t-l+1 \le i \le R).
    \end{align*}
    The first $t-l$ rows of the remaining matrix have entries only in the columns $R+1, \ldots, R+t-l$. We may apply elementary row operations to diagonalise the submatrix $[1, t-l] \times [R+1, R+t-l]$, and use this diagonal matrix in order to eliminate all entries in the submatrix $ [t-l+1, t] \times [R+1, R+t-l]$. This operation does not affect the matrix $D_1 = [t-l+1, R] \times ([t-l+1, R]\cup [R+t-l+1, S])$. This means that $D_1$ is auxiliary of type $(n-1,r,0)_{r,l}$, and thus the number of solutions of the subsystem associated to the matrix $D_1$ is bounded above by $I_{n-1,r}^0(P)$. It thus remains to bound the number $N_1$ of solutions to the system
    \begin{align*}
        d_{i,R+i}(x_{i,1}^3+x_{i,2}^3-x_{i,3}^3-x_{i,4}^3) &= 0 \qquad (1 \le i \le t-l),
    \end{align*}
    and by Hua's Lemma~\cite[Lemma~2.5]{V:HL} we obtain the bound
    \begin{align*}
        N_1 \ll \prod_{i=1}^{t-l} \int_0^1|g(d_{i, R+i}\eta)|^4 \D \eta \ll P^{2(t-l)+\eps}.
    \end{align*}
    Thus we conclude that $H_0 \ll P^{2(t-l)+\eps} I_{n-1,r}^{0}(P)$, and the corresponding bound for $T_0$ is acceptable.
    \begin{figure}
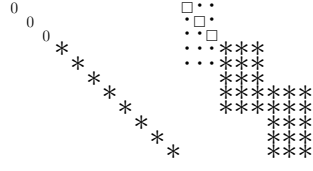

        \begin{align*}
    \arraycolsep=0pt\def\arraystretch{0}
        \begin{array}{*{20}c}
                \sm 0&     &     & & & & & & & & & \ssq&\cdot&\cdot& & & & & & \\
                     &\sm 0&     & & & & & & & & &\cdot& \ssq&\cdot& & & & & & \\
                     &     &\sm 0& & & & & & & & &\cdot&\cdot& \ssq& & & & & & \\
                     &     &     &*& & & & & & & &\cdot&\cdot&\cdot&*&*&*& & & \\
                     &     &     & &*& & & & & & &\cdot&\cdot&\cdot&*&*&*& & & \\
                     &     &     & & &*& & & & & &     &     &     &*&*&*& & & \\
                     &     &     & & & &*& & & & &     &     &     &*&*&*&*&*&*\\
                     &     &     & & & & &*& & & &     &     &     &*&*&*&*&*&*\\
                     &     &     & & & & & &*& & &     &     &     & & & &*&*&*\\
                     &     &     & & & & & & &*& &     &     &     & & & &*&*&*\\
                     &     &     & & & & & & & &*&     &     &     & & & &*&*&*\\
              \hphantom{*} &\hphantom{*} &\hphantom{*} &
     \end{array}
    \end{align*}
    \caption{\footnotesize{Schematic illustration of the argument of bounding $H_0$ for an auxiliary matrix of type $(3,5,0)_{5,2}$. The entries where $x_{i,1}= x_{i,2}$ have been marked with zeros. After diagonalising the corresponding entries on the right hand side, the residual matrix $D_1$, marked by asterisks, is auxiliary of type $(2,5,0)_{5,2}$. }}\label{f-ind1}
    \end{figure}

    We now turn to the contribution of $T_j$ for $1 \le j \le t-l$. By symmetry, it is enough to consider the case $j=1$. Denote by $c_h$ the number of integral solutions $-P \le x,y \le P$ to the equation $d_{1,1}(x^3-y^3)=h$ and write
    \begin{align}\label{T(h)}
        T(h)= \oint \prod_{i=2}^R|g(\theta_i)|^2 \prod_{i=R+1}^S |g(\theta_i)|^4 e(\eta_1 h)\D {\bm \eta},
    \end{align}
    then we find that
    \begin{align*}
        T_1 = \sum_{h \in \Z \setminus \{0\}} c_h T(h).
    \end{align*}
    Observe that $T(h)=0$ except when $|h| \ll P^4$. Furthermore, it follows from an elementary divisor estimate that $c_h \ll h^\eps$ for all $h\ne 0$. We therefore deduce that
    \begin{align*}
        T_1 \ll P^\eps \sum_{h \in \Z} T(h),
    \end{align*}
    and it follows from considering the underlying equations that $\sum_{h \in \Z}T(h)$ counts the number of the solutions of the system associated to the matrix $[2,R] \times [2,S]$. When $ t \ge l+1$, this matrix is auxiliary of type $(n, t-1, 1)_{r,l}$. This concludes the proof of the lemma.
\end{proof}

\begin{lem}\label{ind3}
    We have
    \begin{align*}
         I_{n,t}^{\omega}(P) \ll P^{3+\eps}( I_{n,t}^{\omega-1}(P) +I_{n,t-1}^{\omega}(P))
    \end{align*}
    for all $n \ge 1$, $t \ge l+1$ and $1 \le \omega \le l$.
\end{lem}

\begin{pf}
    Let $D$ be an auxiliary matrix of type $(n,t,\omega)_{r,l}$. Then $I(P,D)$ counts the number of solutions to a system of the shape \eqref{ind1-sys}. By subtracting multiples of the first equation from the lower ones, we can eliminate the entries $d_{i, R+1}$ for $2 \le  i \le t$, so that in \eqref{chvar} we get $\theta_{R+1} = d_{1, R+1} \eta_1$.

    For a measurable subset $\F B \subseteq \T^R$ write
    \begin{align}\label{N(B)}
        I(P,D; \F B) = \int_{\F B} \prod_{i=1}^R |g(\theta_i)|^2 \prod_{i=R+1}^S |g(\theta_i)|^4 \D \bm \eta.
    \end{align}
    Denote by $\F M$ the union of intervals
    \begin{align}\label{M-def}
        \F M(q,a) = \{\eta \in \T: |q \eta-a| \le P^{-9/4}\}
    \end{align}
    with $1 \le a \le q \le P^{3/4}$, and set $\F m = \T \setminus \F M$. Then for every fixed non-zero integer $c$ we have the bounds
    \begin{align}\label{M-bd}
        \int_{\F M}|g(c \eta)|^4 \D \eta \ll P^{1+\eps}
    \end{align}
    and
    \begin{align}\label{m-bd}
        \sup_{\eta \in \F m}|g(c \eta)|  \ll P^{3/4+\eps},
    \end{align}
    stemming from \cite[Lemma~3.4]{BKW:01cubes} and \cite[Lemma~1]{Va:86a}, respectively. Write further $\F n$ for the set of those $\bm \eta \in \T^R$ having $\eta_{1} \in \F m$, and $\F N$ for the complementary set having $\eta_{1} \in \F M $.  In this notation we clearly have
    \begin{align*}
        I(P,D) \ll I(P,D; \F n)+ I(P,D; \F N),
    \end{align*}
    and it follows immediately from \eqref{m-bd} that we have the bound
    \begin{align*}
        I(P,D; \F n) \ll P^{3+\eps} \oint \prod_{i=1}^R |g(\theta_i)|^2 \prod_{i=R+2}^S |g(\theta_i)|^4 \D {\bm \eta}.
    \end{align*}
    Denote by $D_1$ the matrix $[1,R] \times ([1, R] \cup [R+2, S])$ occurring in this mean value. By reversing our inital elementary row operations we see that $D_1$ is row-equivalent to an auxiliary matrix of type $(n,t,\omega-1)_{r,l}$. It thus follows from considering the underlying equations that $I(P, D_1) \ll I_{n,t}^{\omega-1}(P)$, whence we obtain the bound $I(P,D; \F n) \ll P^{3+\eps}I_{n,t}^{\omega-1}(P) $.

    It remains to estimate the contribution from $\F N$. Observe that the rows $2, \ldots, R$ are populated only in the columns $1, \ldots, R$ and $R+2, \ldots, S$. Write $\bm \eta' = (\eta_{2}, \ldots, \eta_R)$ and
    \begin{align*}
        \F G( \eta_1 ) =\oint \prod_{i=2}^R |g(\theta_i)|^2 \prod_{i=R+2}^S |g(\theta_i)|^4 \D \bm \eta',
    \end{align*}
    then estimating the exponential sum $g(\theta_1)$ trivially yields
    \begin{align}\label{ind3-IG}
        I(P,D; \F N) \ll P^2 \int_{\F M}  |g(d_{1, R+1}\eta_{1})|^4 \F G(\eta_1 )  \D \eta_1.
    \end{align}
    The function $\F G(\eta_1)$ counts the number of solutions to the system of equations given by the matrix $D_2 = [2,R] \times ([2,R] \cup [R+2,S])$ equipped with a unimodular weight depending on $\eta_1 $. It therefore follows by the triangle inequality that $|\F G(\eta_1 )| \le \F G( 0)$. Substituting this in \eqref{ind3-IG} produces the estimate
    \begin{align*}
        I(P,D; \F N) \ll  P^{2} \F G( 0) \int_{\F M} |g(d_{1, R+1}\eta_{1})|^4  \D \eta_1 \ll P^{3+\eps} \F G( 0),
    \end{align*}
    where in the last step we applied \eqref{M-bd}.

    It remains to show that the matrix $D_2$ is auxiliary of type $(n,t-1,\omega)_{r,l}$. In order to see this, we only need to check that the submatrix $M=[2, t] \times [R+2, R+t-l+\omega]$ of $D_2$ is totally non-singular. This matrix has been obtained from the totally non-singular submatrix $L= [1, t] \times [R+1, R+t-l+\omega]$. Since $L$ is totally non-singular, the matrix $L^*=(\Id_t , L)$ is highly non-singular by Lemma~\ref{hns-prop} (ii), and this property is not affected by elementary row operations. If we thus use the top left element of $L$ to eliminate all other entries of the first column of $L$, the correspondingly transformed matrix $L^*$ is still highly non-singular. By Lemma~\ref{hns-prop} (i) we may now eliminate the first and $(R+1)$-st columns and the first row of this transformed matrix without losing high nonsingularity. The resulting matrix is of the shape $(\Id_{t-1}, M)$, so $M$ is totally non-singular by Lemma~\ref{hns-prop} (ii). It follows that the matrix $D_2$ is indeed auxiliary of type $(n,t-1,\omega)_{r,l}$. This allows us to bound $\F G( 0) \ll I_{n,t-1}^{\omega}(P)$, which completes the proof of the lemma.
\end{pf}

\begin{lem}\label{ind4}
    For all $1 \le \omega \le l$ and $n\ge 2$ we have
    \begin{align*}
         I_{n,l}^{\omega}(P) \ll P^{3+\eps} I_{n,l}^{\omega-1}(P) + \sum_{m=0}^{l-\omega}  P^{3 \omega + 3m+\eps} I_{n-1, r-l}^{l-m}(P).
    \end{align*}
\end{lem}

\begin{pf}
    Suppose that $D$ is an auxiliary matrix of type $(n,l,\omega)_{r,l}$. As in the previous lemmas, understanding the mean value $I(P,D)$ is tantamount to counting the number of solutions to a system of equations of the shape \eqref{ind1-sys}.
    Since the matrix $[1, \omega] \times [R+1, R+\omega]$ is non-singular, we can take linear combinations of the first $l$ rows in order to diagonalise this matrix and eliminate all entries of $[\omega+1,R] \times [R+1, R+\omega]$. This operation leaves the diagonal matrix $[\omega+1,R] \times [\omega+1, R]$ intact and simultaneously allows us to write $\theta_{R+j} = d_{j, R+j} \eta_j$ for $1 \le j \le \omega$ in \eqref{chvar}. Recall the definition of the major and minor arcs from \eqref{M-def}. For $1 \le j \le \omega$ write $\cal B_j$ for the set of $\bm \eta \in \T^R$ with $\eta_{j} \in \F m $, and let $\cal B_0$ denote the complementary set where $\eta_{j} \in \F M$ for $1 \le j \le \omega$. This implies that we have
    \begin{align*}
        I(P,D) \ll I(P,D;\cal B_0) +I(P,D;\cal B_1)  + \ldots + I(P,D;\cal B_\omega),
    \end{align*}
    where we used the notation introduced in \eqref{N(B)}. Just like in the previous lemma we derive from \eqref{m-bd} the bound
    \begin{align*}
        I(P,D;\cal B_j) \ll P^{3+\eps} I_{n,l}^{\omega-1}(P)
    \end{align*}
    for $1 \le j \le \omega$.
    It thus suffices to study the contribution from $\cal B_0$. Write $\bm \eta' = (\eta_1, \ldots, \eta_{\omega})$ and $\bm \eta^*=(\eta_{\omega+1}, \ldots, \eta_R)$, and let
    \begin{align*}
        \F G(\bm \eta') = \oint \prod_{i=\omega+1}^R |g(\theta_i)|^2 \prod_{i=R+\omega+1}^S |g(\theta_i)|^4 \D \bm \eta^*.
    \end{align*}
    Then after estimating the first $\omega$ exponential sums trivially, we arrive at the bound
    \begin{align*}
        I(P,D;\cal B_0) \ll  P^{2\omega} \int_{\F M^\omega} \prod_{i=R+1}^{R+\omega}|g(d_{i, R+i} \eta_i)|^4 \F G(\bm \eta') \D \bm \eta'.
    \end{align*}
    The function $\F G(\bm \eta')$ counts the number of solutions to the system
    \begin{align}\label{ind4-sys}
        d_{i,i}(x_{i,1}^3 - x_{i,2}^3) + \sum_{j=R+\omega+1}^S d_{i,j}(x_{j,1}^3+x_{j,2}^3-x_{j,3}^3-x_{j,4}^3) &= 0 \qquad (\omega+1 \le i \le R)
    \end{align}
    associated to the matrix $[\omega+1, R] \times ([\omega+1, R] \cup [R+\omega+1, S])$, where each solution carries a unimodular weight depending on $\bm \eta'$, and it follows from the triangle inequality that $|\F G(\bm \eta')| \le \F G(\bm 0)$. Hence by applying \eqref{M-bd}, we find that
    \begin{align}\label{ind4-IG}
        I(P,D;\cal B_0) \ll  P^{2\omega} \F G(\bm 0) \int_{\F M^\omega} \prod_{i=1}^{\omega}|g(d_{i, R+i}\eta_i)|^4  \D \bm \eta' \ll P^{3\omega+\eps}\F G(\bm 0).
    \end{align}
    Our task is therefore to bound the exponential sum $\F G(\bm 0)$. However, since we have performed elementary row operations on the first $l$ rows, the matrix associated to this system is not necessarily auxiliary, as is illustrated in Figure \ref{f-ind4a}. This forces us to be quite careful in our operations. In particular, it does not allow us to estimate $\F G(\bm 0)$ by $I_{n-1, r-\omega}^{\omega}(P)$, as might be desirable.

    \begin{figure}[H]
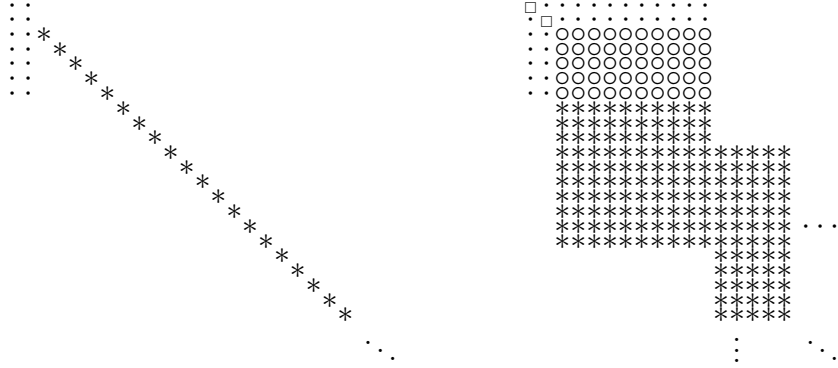

        \begin{align*}%
            \arraycolsep=0pt\def\arraystretch{0}
            \begin{array}{*{42}c}
                \cdot&\cdot& & & & & & & & & & & & & & & & & & & & & & \ssq&\cdot&\cdot&\cdot&\cdot&\cdot&\cdot&\cdot&\cdot&\cdot&\cdot&\cdot& & & & & & &\\
                \cdot&\cdot& & & & & & & & & & & & & & & & & & & & & &\cdot& \ssq&\cdot&\cdot&\cdot&\cdot&\cdot&\cdot&\cdot&\cdot&\cdot&\cdot& & & & & & &\\
                \cdot&\cdot&*& & & & & & & & & & & & & & & & & & & & &\cdot&\cdot&\circ&\circ&\circ&\circ&\circ&\circ&\circ&\circ&\circ&\circ& & & & & & &\\
                \cdot&\cdot& &*& & & & & & & & & & & & & & & & & & & &\cdot&\cdot&\circ&\circ&\circ&\circ&\circ&\circ&\circ&\circ&\circ&\circ& & & & & & &\\
                \cdot&\cdot& & &*& & & & & & & & & & & & & & & & & & &\cdot&\cdot&\circ&\circ&\circ&\circ&\circ&\circ&\circ&\circ&\circ&\circ& & & & & & &\\
                \cdot&\cdot& & & &*& & & & & & & & & & & & & & & & & &\cdot&\cdot&\circ&\circ&\circ&\circ&\circ&\circ&\circ&\circ&\circ&\circ& & & & & & &\\
                \cdot&\cdot& & & & &*& & & & & & & & & & & & & & & & &\cdot&\cdot&\circ&\circ&\circ&\circ&\circ&\circ&\circ&\circ&\circ&\circ& & & & & & &\\
                     &     & & & & & &*& & & & & & & & & & & & & & & &     &     &  *  &  *  &  *  &  *  &  *  &  *  &  *  &  *  &  *  &  *  & & & & & & &\\
                     &     & & & & & & &*& & & & & & & & & & & & & & &     &     &  *  &  *  &  *  &  *  &  *  &  *  &  *  &  *  &  *  &  *  & & & & & & &\\
                     &     & & & & & & & &*& & & & & & & & & & & & & &     &     &  *  &  *  &  *  &  *  &  *  &  *  &  *  &  *  &  *  &  *  & & & & & & &\\
                     &     & & & & & & & & &*& & & & & & & & & & & & &     &     &  *  &  *  &  *  &  *  &  *  &  *  &  *  &  *  &  *  &  *  &*&*&*&*&*& &\\
                     &     & & & & & & & & & &*& & & & & & & & & & & &     &     &  *  &  *  &  *  &  *  &  *  &  *  &  *  &  *  &  *  &  *  &*&*&*&*&*& &\\
                     &     & & & & & & & & & & &*& & & & & & & & & & &     &     &  *  &  *  &  *  &  *  &  *  &  *  &  *  &  *  &  *  &  *  &*&*&*&*&*& &\\
                     &     & & & & & & & & & & & &*& & & & & & & & & &     &     &  *  &  *  &  *  &  *  &  *  &  *  &  *  &  *  &  *  &  *  &*&*&*&*&*& &\\
                     &     & & & & & & & & & & & & &*& & & & & & & & &     &     &  *  &  *  &  *  &  *  &  *  &  *  &  *  &  *  &  *  &  *  &*&*&*&*&*& &\\
                     &     & & & & & & & & & & & & & &*& & & & & & & &     &     &  *  &  *  &  *  &  *  &  *  &  *  &  *  &  *  &  *  &  *  &*&*&*&*&*&\;\cdots  &\\
                     &     & & & & & & & & & & & & & & &*& & & & & & &     &     &  *  &  *  &  *  &  *  &  *  &  *  &  *  &  *  &  *  &  *  &*&*&*&*&*& &\\
                     &     & & & & & & & & & & & & & & & &*& & & & & &     &     &     &     &     &     &     &     &     &     &     &     &*&*&*&*&*& &\\
                     &     & & & & & & & & & & & & & & & & &*& & & & &     &     &     &     &     &     &     &     &     &     &     &     &*&*&*&*&*& &\\
                     &     & & & & & & & & & & & & & & & & & &*& & & &     &     &     &     &     &     &     &     &     &     &     &     &*&*&*&*&*& &\\
                     &     & & & & & & & & & & & & & & & & & & &*& & &     &     &     &     &     &     &     &     &     &     &     &     &*&*&*&*&*& &\\
                     &     & & & & & & & & & & & & & & & & & & & &*& &     &     &     &     &     &     &     &     &     &     &     &     &*&*&*&*&*& &\\
                 \hphantom{*} &\hphantom{*} &\hphantom{*} &\hphantom{*} &\hphantom{*} &\hphantom{*} &\hphantom{*} &\hphantom{*} &\hphantom{*} &\hphantom{*} &\hphantom{*} &\hphantom{*} &\hphantom{*} &\hphantom{*} &\hphantom{*} &\hphantom{*} &\hphantom{*} &\hphantom{*} &\hphantom{*} &\hphantom{*} &\hphantom{*} &\hphantom{*} &\;\ddots \qquad\qquad &\hphantom{*} &\hphantom{*} &\hphantom{*} &\hphantom{*} &\hphantom{*} &\hphantom{*} &\hphantom{*} &\hphantom{*} &\hphantom{*} &\hphantom{*} &\hphantom{*} &\hphantom{*} &\hphantom{*} & \vdots &\hphantom{*} &\hphantom{*} &\hphantom{*} &\;\ddots
            \end{array}
        \end{align*}
        \caption{\footnotesize{A schematic representation of the first two and a half blocs of $D$ after \eqref{ind4-IG}, with parameters $r=17$, $l=7$, $\omega=2$. The columns $R+1, \ldots, R+\omega$ have been diagonalised, and the first two columns and rows are deleted in the estimate \eqref{ind4-IG}. The matrix associated to $\F G(\bm 0)$ is marked with circles and asterisks, and all entries affected by the elementary row operations have been marked by a circle. }}\label{f-ind4a}
    \end{figure}

    Let $T_m$ denote the number of solutions to \eqref{ind4-sys} where $x_{i,1} = x_{i,2}$ for precisely $m$ indices $\omega+1 \le i \le l$, and we may assume without loss of generality that these are the indices $\omega+1 \le i \le \omega+m$. In this notation we have
    \begin{align}\label{ind4-GT}
        \F G(\bm 0) \ll \sum_{m=0}^{l-\omega} T_m.
    \end{align}
    For each $T_m$, there are $(2P+1)^m$ possible choices for the variables $x_{i,k}$ with $\omega+1 \le i \le \omega+m$ and $k \in \{1,2\}$. It follows that $T_m \ll P^m I(P,D_m)$, where $D_m$ denotes the matrix $[\omega+1, R] \times ([\omega+m +1, R] \cup [R+\omega+1, S])$. Since the submatrix $[1, l] \times [R+\omega+1, R+\omega +r-l]$ of $D$ had been of rank $l$, even after performing elementary row operations and deleting the first $\omega$ rows, its lower $l-\omega$ rows are still of full rank. This allows us to assume, without loss of generality, that the matrix $[\omega+1, \omega+m] \times [R+\omega+1, R+\omega +m]$ is non-singular and can thus be diagonalised.
    Write now $\bm \eta_1^* = (\eta_{\omega+1}, \ldots, \eta_{\omega+m})$ and $\bm \eta_2^* = (\eta_{\omega+m+1}, \ldots, \eta_{R})$, and let
    \begin{align*}
        \F H_m(\bm \eta^*_1) = \oint \prod_{i=\omega+m+1}^R |g(\theta_i)|^2 \prod_{i=R+\omega+m+1}^S |g(\theta_i)|^4 \D \bm \eta^*_2.
    \end{align*}
    In this notation we have
    \begin{align*}
        T_m & \ll  P^m\oint \sup_{ \bm \eta_2^*} \left(\prod_{i=R+\omega+1}^{R+ \omega+m} |g(\theta_i)|^4\right) \F H_m(\bm \eta^*_1)\D \bm \eta_1^*,
    \end{align*}
    where $\F H_m(\bm \eta^*_1)$ counts the solutions to the system associated to the matrix
    \begin{align*}
        D_m^*=[\omega+m+1, R] \times([\omega+m+1, R] \cup [R+\omega+m+1, S]),
    \end{align*}
    again weighted by a unimodular weight depending on $\bm \eta^*_1$. As before, the triangle inequality allows us to simplify $|\F H_m(\bm \eta^*_1)| \le \F H_m(\bm 0)$, and by a similar argument we see that the supremum over $\bm \eta_2^*$ is taken at $\bm \eta_2^*= \bm 0$. Since the matrix $[\omega+1, \omega+m] \times [R+\omega+1, R+\omega +m]$ had been diagonalised, we may conclude that
    \begin{align}\label{ind4-TH}
        T_m & \ll  P^m\F H_m(\bm 0) \oint \prod_{i=\omega+1}^{\omega+m} |g(d_{i, R+i}\eta_i)|^4 \D \bm \eta_1^*  \ll P^{3m+\eps}\F H_m(\bm 0),
    \end{align}
    where in the last step we applied Hua's Lemma~\cite[Lemma~2.5]{V:HL}.
    \begin{figure}
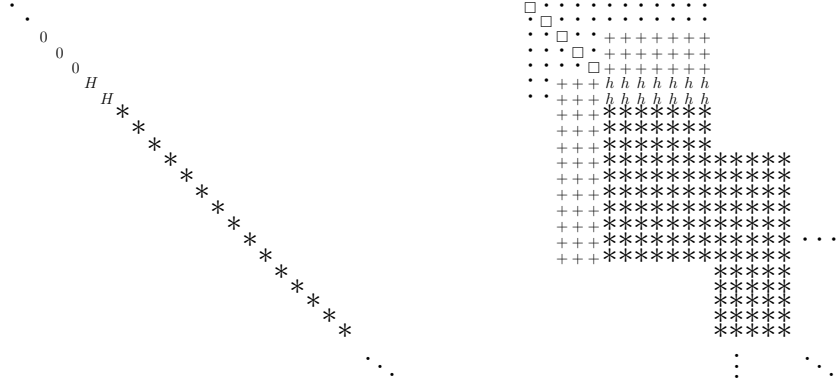

        \begin{align*}
            \arraycolsep=0pt\def\arraystretch{0}
            \begin{array}{*{43}c}
                \cdot&     &     &     &     &     &     & & & & & & & & & & & & & & & & & \ssq&\cdot&\cdot&\cdot&\cdot&\cdot&\cdot&\cdot&\cdot&\cdot&\cdot&\cdot& & & & & &\phantom{*} \\
                     &\cdot&     &     &     &     &     & & & & & & & & & & & & & & & & &\cdot& \ssq&\cdot&\cdot&\cdot&\cdot&\cdot&\cdot&\cdot&\cdot&\cdot&\cdot& & & & & &\phantom{*} \\
                     &     &\sm 0&     &     &     &     & & & & & & & & & & & & & & & & &\cdot&\cdot& \ssq&\cdot&\cdot&\sm +&\sm +&\sm +&\sm +&\sm +&\sm +&\sm +& & & & & &\phantom{*} \\
                     &     &     &\sm 0&     &     &     & & & & & & & & & & & & & & & & &\cdot&\cdot&\cdot& \ssq&\cdot&\sm +&\sm +&\sm +&\sm +&\sm +&\sm +&\sm +& & & & & & \\
                     &     &     &     &\sm 0&     &     & & & & & & & & & & & & & & & & &\cdot&\cdot&\cdot&\cdot& \ssq&\sm +&\sm +&\sm +&\sm +&\sm +&\sm +&\sm +& & & & & & \\
                     &     &     &     &     &\sm H&     & & & & & & & & & & & & & & & & &\cdot&\cdot&\sm +&\sm +&\sm +&\sm h&\sm h&\sm h&\sm h&\sm h&\sm h&\sm h& & & & & & \\
                     &     &     &     &     &     &\sm H& & & & & & & & & & & & & & & & &\cdot&\cdot&\sm +&\sm +&\sm +&\sm h&\sm h&\sm h&\sm h&\sm h&\sm h&\sm h& & & & & & \\
                     &     &     &     &     &     &     &*& & & & & & & & & & & & & & & &     &     &\sm +&\sm +&\sm +&  *  &  *  &  *  &  *  &  *  &  *  &  *  & & & & & & \\
                     &     &     &     &     &     &     & &*& & & & & & & & & & & & & & &     &     &\sm +&\sm +&\sm +&  *  &  *  &  *  &  *  &  *  &  *  &  *  & & & & & & \\
                     &     &     &     &     &     &     & & &*& & & & & & & & & & & & & &     &     &\sm +&\sm +&\sm +&  *  &  *  &  *  &  *  &  *  &  *  &  *  & & & & & & \\
                     &     &     &     &     &     &     & & & &*& & & & & & & & & & & & &     &     &\sm +&\sm +&\sm +&  *  &  *  &  *  &  *  &  *  &  *  &  *  &*&*&*&*&*& \\
                     &     &     &     &     &     &     & & & & &*& & & & & & & & & & & &     &     &\sm +&\sm +&\sm +&  *  &  *  &  *  &  *  &  *  &  *  &  *  &*&*&*&*&*& \\
                     &     &     &     &     &     &     & & & & & &*& & & & & & & & & & &     &     &\sm +&\sm +&\sm +&  *  &  *  &  *  &  *  &  *  &  *  &  *  &*&*&*&*&*& \\
                     &     &     &     &     &     &     & & & & & & &*& & & & & & & & & &     &     &\sm +&\sm +&\sm +&  *  &  *  &  *  &  *  &  *  &  *  &  *  &*&*&*&*&*& \\
                     &     &     &     &     &     &     & & & & & & & &*& & & & & & & & &     &     &\sm +&\sm +&\sm +&  *  &  *  &  *  &  *  &  *  &  *  &  *  &*&*&*&*&*& \\
                     &     &     &     &     &     &     & & & & & & & & &*& & & & & & & &     &     &\sm +&\sm +&\sm +&  *  &  *  &  *  &  *  &  *  &  *  &  *  &*&*&*&*&*& \;\cdots\\
                     &     &     &     &     &     &     & & & & & & & & & &*& & & & & & &     &     &\sm +&\sm +&\sm +&  *  &  *  &  *  &  *  &  *  &  *  &  *  &*&*&*&*&*& \\
                     &     &     &     &     &     &     & & & & & & & & & & &*& & & & & &     &     &     &     &     &     &     &     &     &     &     &     &*&*&*&*&*& \\
                     &     &     &     &     &     &     & & & & & & & & & & & &*& & & & &     &     &     &     &     &     &     &     &     &     &     &     &*&*&*&*&*& \\
                     &     &     &     &     &     &     & & & & & & & & & & & & &*& & & &     &     &     &     &     &     &     &     &     &     &     &     &*&*&*&*&*& \\
                     &     &     &     &     &     &     & & & & & & & & & & & & & &*& & &     &     &     &     &     &     &     &     &     &     &     &     &*&*&*&*&*& \\
                     &     &     &     &     &     &     & & & & & & & & & & & & & & &*& &     &     &     &     &     &     &     &     &     &     &     &     &*&*&*&*&*& \\
                 \hphantom{*} &\hphantom{*} &\hphantom{*} &\hphantom{*} &\hphantom{*} &\hphantom{*} &\hphantom{*} &\hphantom{*} &\hphantom{*} &\hphantom{*} &\hphantom{*} &\hphantom{*} &\hphantom{*} &\hphantom{*} &\hphantom{*} &\hphantom{*} &\hphantom{*} &\hphantom{*} &\hphantom{*} &\hphantom{*} &\hphantom{*} &\hphantom{*} &\;\ddots \qquad\qquad &\hphantom{*} &\hphantom{*} &\hphantom{*} &\hphantom{*} &\hphantom{*} &\hphantom{*} &\hphantom{*} &\hphantom{*} &\hphantom{*} &\hphantom{*} &\hphantom{*} &\hphantom{*} &\hphantom{*} & \vdots &\hphantom{*} &\hphantom{*} &\hphantom{*} &\;\ddots
            \end{array}
        \end{align*}
        \caption{\footnotesize{Schematic representation of the arguments around \eqref{ind4-TH}, \eqref{ind4-HR} and \eqref{ind4-HI} with $m=3$. The zeros denote the entries with $x_{i,1}=x_{i,2}$; the corresponding entries on the right hand side have been diagonalised in \eqref{ind4-TH}.  It then follows from \eqref{ind4-TH} that the entries marked by $+$ signs can be neglected. The matrix associated to $\F H_m(\bm 0)$ is the one marked by asterisks and the letters $h$ and $H$. The letter $H$ corresponds to the entries with $x_{i,1}- x_{i,2}=h_i \neq 0$; the corresponding rows, marked with $h$, are deleted by the argument around \eqref{ind4-HR}. The residual matrix $D_m^{\dagger}$, marked with asterisks, has not been affected by any of the row operations and is therefore auxiliary.}}\label{f-ind4b}
    \end{figure}

    By our definition, $T_m$ counts the number of solutions having $x_{i,1} = x_{i,2}$ precisely for the indices $\omega+1 \le i \le \omega+m$, so we may assume that $x_{i,1} \neq x_{i,2}$ for $\omega+m+1 \le i \le l$. Let
    \begin{align*}
        R_m(\B h) = \oint \prod_{i=l+1}^R |g(\theta_i)|^2  \prod_{i=R+\omega+m+1}^S |g(\theta_i)|^4 \prod_{i=\omega+m+1}^l e(h_i \eta_i) \D \bm \eta,
    \end{align*}
    then we have
    \begin{align}\label{ind4-HR}
        \F H_m(\bm 0) \ll \sum_{\substack{h_i \in \Z \setminus \{0\} \\ \omega + m +1 \le i \le l}} c_{\B h} R_m(\B h),
    \end{align}
    where $c_{\B h}$ denotes the number of solutions $-P \le x_{i,1}, x_{i,2} \le P$ to $d_{i,i}(x_{i,1}^3 - x_{i,2}^3) = h_i$ for $\omega + m +1 \le i \le l$.
    It follows from a divisor estimate that $c_{\B h} \ll |h_{\omega+m+1} \cdot \ldots \cdot h_l|^{\eps}$, and since $R_m(\B h) = 0$ for $\max |h_{i}| \gg P^{4}$, we obtain the bound
    \begin{align*}
        \F H_m(\bm 0) \ll P^{\eps} \sum_{\B h} R_m(\B h).
    \end{align*}
    On considering the underlying system of equations, we see that the sum $\sum_{\B h} R_m(\B h)$ counts the number of solutions to the system associated with the matrix
    \begin{align*}
        D_m^\dagger =  [l+1, R] \times([l+1, R] \cup [R+\omega+m+1, S]).
    \end{align*}
    This matrix is now auxiliary of type $(n-1, r-l, l-m)_{r,l}$. In order to see this, we need to show that the matrix $[l+1, r] \times [R+\omega+m+1, R+\omega+r-l]$ is totally non-singular. However, this follows directly upon observing that this submatrix has been obtained from the totally non-singular matrix $[1, r] \times [R+\omega+1, R+\omega+r-l]$ by deleting the first $l$ rows and the first $m$ columns. This has been illustrated in Figure \ref{f-ind4b}.
    It therefore follows that we may estimate
    \begin{align}\label{ind4-HI}
        \F H_m(\bm 0) \ll P^{\eps} I_{n-1, r-l}^{l-m}(P).
    \end{align}
    The statement of the Lemma~now follows upon combining the statements \eqref{ind4-IG}, \eqref{ind4-GT}, \eqref{ind4-TH} and \eqref{ind4-HI}.
\end{pf}

\begin{lem}\label{ind0}
    We have the bound $I_{1,l}^{0}(P) \ll P^l$. Furthermore, we have
    \begin{align*}
        I_{1,t}^{0}(P) &\ll P^{3t-2l+\eps} + P^\eps I_{1,t-1}^{1}(P) \quad \text{for $t \ge l+1$, and}\\
        I_{1,l}^{\omega}(P) &\ll P^{l+2\omega+\eps} + P^{3+\eps}I_{1,l}^{\omega-1}(P) \quad \text{for $\omega \ge 1$.}
    \end{align*}
\end{lem}

\begin{pf}
    The quantity $I_{1,l}^{0}(P)$ counts solutions to the system
    \begin{align*}
        d_{i,i} (x_{i,1}^3-x_{i,2}^3) = 0 \qquad (1 \le i \le l)
    \end{align*}
    for some non-zero coefficients $d_{i,i}$. Obviously, the number of solutions to this system is precisely $(2P+1)^l$.

    For the second statement we proceed as in Lemma~\ref{ind1}. Let $D_1$ be auxiliary of type $(1,t,0)_{r,l}$, then $I(P,D_1)$ describes the number of solutions $-P \le x_{j,k} \le P$ to the system
    \begin{align}\label{ind0-sys}
        d_{i,i}(x_{i,1}^3 - x_{i,2}^3) + \sum_{j=t+1}^{2t-l} d_{i,j}(x_{j,1}^3+x_{j,2}^3-x_{j,3}^3-x_{j,4}^3) = 0 \qquad (1 \le i \le t).
    \end{align}
    Write $T_0$ for the number of solutions counted by \eqref{ind0-sys} having $x_{j,1}=x_{j,2}$ for all $1 \le j \le t-l$, and denote by $T_j$  the number of solutions having $x_{j,1} \neq x_{j,2}$. Then we have
    \begin{align*}
        I(P, D_1) \le T_0 + T_1 + \ldots + T_{t-l},
    \end{align*}
    and one sees easily that $T_0 \ll P^{t-l}H_0$, where $H_0$ denotes the number of solutions to the system
    \begin{align*}
        \sum_{j=t+1}^{2t-l} d_{i,j}(x_{j,1}^3+x_{j,2}^3-x_{j,3}^3-x_{j,4}^3) &= 0 \qquad (1 \le i \le t-l), \\
        d_{i,i}(x_{i,1}^3 - x_{i,2}^3) + \sum_{j=t+1}^{2t-l} d_{i,j}(x_{j,1}^3+x_{j,2}^3-x_{j,3}^3-x_{j,4}^3) &= 0 \qquad (t-l+1 \le i \le t).
    \end{align*}
    We may apply elementary row operations to diagonalise the submatrix $[1, t-l] \times [t+1,2t-l]$, and use this diagonal matrix in order to eliminate all entries in the submatrix $ [t-l+1, t] \times [t+1,2t-l]$.
    The rows $t-l+1, \ldots, t$ are now empty but for the diagonal matrix $[t-l+1, t] \times [t-l+1, t]$, and thus correspond to the system of equations
    \begin{align*}
        d_{i,i}(x_{i,1}^3 - x_{i,2}^3) = 0 \qquad (t-l+1 \le i \le t)
    \end{align*}
    having precisely $(2P+1)^{l}$ solutions. It remains to bound the number $N$ of solutions to the system corresponding to the matrix $[1,t-l] \times [t+1,2t-l]$ consisting of the first $t-l$ rows of $D_1$. This matrix is now diagonal, and it follows from Hua's Lemma~\cite[Lemma~2.5]{V:HL} that
    \begin{align*}
        N \ll \prod_{i=1}^{t-l} \int_0^1|g(d_{i, t+i}\eta)|^4 \D \eta \ll P^{2(t-l)+\eps}.
    \end{align*}
    Thus we conclude that $H_0 \ll P^{2(t-l)+\eps} P^{l}$ and thus $T_0 \ll P^{3t-2l+\eps}$, which is in accordance with the enunciation of the lemma.
    In order to bound $T_j$ for $j \ge 1$, it remains to observe that the argument around equation \eqref{T(h)} of Lemma~\ref{ind1} applies unchanged and leads to the bound $T_j \ll P^{\eps}I_{1,t-1}^{1}(P)$. This proves the second statement of the lemma.

    It remains to establish the third statement. Let $D_2$ be auxiliary of type $(1,l,\omega)_{r,l}$, and recall the definition of the major and minor arcs from \eqref{M-def}. Since the submatrix $[1, \omega] \times[l+1, l+\omega]$ of $D_2$ is non-singular, we may diagonalise it and use it to eliminate all entries in $[\omega+1, l] \times [l+1, l+\omega]$. This allows us to write $\theta_{l+j} = d_{j, l+j} \eta_j$ for $1 \le j \le \omega$ in \eqref{chvar}, while leaving the diagonal matrix $[\omega+1, l] \times [\omega+1, l]$ intact. For $1 \le j \le \omega$ write $\cal B_j$ for the set of $\bm \eta \in \T^l$ with $\eta_{j} \in \F m $, and let $\cal B_0$ denote the complementary set where $\eta_{j} \in \F M$ for $1 \le j \le \omega$. Then we have
    \begin{align*}
        I(P,D_2) \ll I(P,D_2; \cal B_0) + I(P,D_2; \cal B_1) + \ldots + I(P,D_2; \cal B_\omega),
    \end{align*}
    and just like in Lemma~\ref{ind3} it follows from \eqref{m-bd} that $I(P,D_2; \cal B_j) \ll P^{3+\eps} I_{1,l}^{\omega-1}(P)$ for $1 \le j \le \omega$.
    It thus remains to bound the contribution from the major arcs $\cal B_0$. Observe that the matrix $[1,l] \times [\omega+1, l+\omega]$ decomposes into two diagonal matrices in $[\omega+1,l] \times [\omega+1,l]$ and $[1, \omega] \times[l+1, l+\omega]$. Hence estimating the exponential sums $g(\theta_1), \ldots, g(\theta_\omega)$ trivially leads to the bound
    \begin{align*}
        I(P,D_2; \cal B_0) &\ll P^{2\omega}\left(\prod_{i=\omega+1}^l \oint  |g(d_{i,i}\eta)|^2 \D \eta \right) \left(\prod_{i=1}^\omega \int_{\F M}|g(d_{i, l+i}\eta)|^4  \D \eta \right)\\
        & \ll P^{2\omega} P^{l-\omega} P^{\omega+\eps} \ll P^{l+2\omega+\eps},
    \end{align*}
    where we applied \eqref{M-bd} in the second step. This completes the proof of the lemma.
\end{pf}

\begin{proof}[Proof of Proposition \ref{silly}]
    This is now swiftly completed and follows from an inductive argument using Lemmata \ref{ind1}, \ref{ind3} and \ref{ind4}. The basis for this induction is provided by Lemma~ \ref{ind0}, which together with Lemma~\ref{ind3} establishes the hypothesis for all auxiliary matrices of type $(1,t,\omega)_{r,l}$.

    The inductive step decomposes into an outer induction on $n$ and an inner induction on $t$ and $\omega$. For the outer induction we observe that an auxiliary matrix of type $(n, l, 0)_{r,l}$ can also be viewed as being of type $(n-1, r,0)_{r,l}$. Since also
    \begin{align*}
        3((n-1)(r-l)+l)-2l = 3((n-2)(r-l)+r)-2l,
    \end{align*}
    it follows that the inductive hypothesis holds for all auxiliary matrices of type $(n,l,0)_{r,l}$ whenever it holds for all auxiliary matrices of type $(n-1, r,0)_{r,l}$. Furthermore, if the inductive hypothesis is known for all matrices of type $(n',t',\omega')_{r,l}$ having either $n' < n$ or $n'=n$, $t'=l$ and $\omega'<\omega$, Lemma~ \ref{ind4} shows that it also holds for all auxiliary matrices of type $(n,l,\omega)_{r,l}$ with $\omega \ge 1$.

    For the inner induction we define an ordering on the pairs $(t,\omega)$ by setting $(t,\omega) \succ (t', \omega')$ if either $t+\omega > t'+\omega'$, or $t+\omega = t'+\omega'$ and $t > t'$. Now suppose that the inductive hypothesis is known for all auxiliary matrices of type $(n',t',\omega')_{r,l}$ having either $n' < n$ or $n'=n$ and $(t',\omega') \prec (t, \omega)$. Then according to the value of $\omega$ one of Lemmata \ref{ind1} and \ref{ind3} is applicable and implies that the desired bound holds for matrices of type $(n,t,\omega)_{r,l}$ as well. This proves the inner inductive step.
\end{proof}

\section{Complification}

In this section we describe the complification process employed in the proof.
Let $n$, $r$ and $l$ be positive integers with $r \ge 2l$, and let $\rho=\rho_n=n(r-l)$. Consider integral matrices $D_n^{(2)}$ and $D_n^{(3)}$ of respective format $l \times (2\rho_n+l)$ and $(\rho_n+l) \times (2\rho_n+l)$, where $D_n^{(2)}$ is highly non-singular and $D_n^{(3)}$ is auxiliary of type $(n,r,0)_{r,l}$. For ease of notation in the following arguments, we will label the rows of the matrix $D_n^{(2)}$ by $\rho+1, \ldots, \rho+l$, so the matrices have column vectors $\B d_j^{(k)} = (d^{(k)}_{i,j})_{i}$ where $\rho+1 \le i \le \rho+l$ if $k=2$ and $1 \le i \le \rho+l$ if $k=3$. Also, define $\gamma_{k,j} = \gamma_{k,j}(\ba)$ by
\begin{align*}
    \gamma_{3,j}(\ba) = \sum_{i=1}^{\rho+l}d_{i,j}^{(3)}\alpha_{3,i} \quad \text{ and } \quad \gamma_{2,j}=  \sum_{i=\rho+1}^{\rho+l}  d_{i,j}^{(2)}\alpha_{2,i} \qquad (1 \le j \le 2\rho+l).
\end{align*}
We abbreviate $\ba_i = (\alpha_{3,i}, \alpha_{2,i})$ for $\rho+1 \le i \le \rho+l$ and write $\ba^{(k)}$ for the vector $(\alpha_{k,i})_i$  with $k \in \{2,3\}$. Furthermore, let $\widetilde{\ba} = (\alpha_{3,1}, \ldots, \alpha_{3,\rho})$ and $\ba^{\dagger} = (\ba_{\rho+1}, \ldots, \ba_{\rho+l})$.

Define the Weyl sum
\begin{align*}
    f(\alpha,\beta) = \sum_{x \in [-P,P]}e(\alpha x^3 + \beta x^2),
\end{align*}
and write $F_{a}^{b}(\ba)= \prod_{i=a}^b f(\bg_i)$. We consider the family of mean values
\begin{align}\label{J-def}
    J_n(P) = \begin{cases}\displaystyle{\oint |F_{1}^{r}(\ba)|^2 |F_{r+1}^{r+l}(\ba)|^{12} |F_{r+l+1}^{2r-l}(\ba)|^4   \D \ba} & \text{ for }n=1,\\ \displaystyle{\oint  |F_{1}^{\rho+l}(\ba)|^2 |F_{\rho+l+1}^{\rho+2l}(\ba)|^8 |F_{\rho+2l+1}^{2\rho}(\ba)|^4 |F_{2\rho+1}^{2\rho+l}(\ba)|^8 \D \ba} & \text{ for }n\ge 2.\end{cases}
\end{align}
For future use we record the trivial inequality
\begin{align}\label{trivineq}
    |x_1 \cdot \ldots \cdot x_n| \le |x_1|^n + \ldots + |x_n|^n
\end{align}
as well as the mean value
\begin{align}\label{mv23}
    \int_0^1 \int_0^1|f(\alpha, \beta)|^{10} \D \alpha \D \beta \ll P^{31/6+\eps}
\end{align}
due to Wooley \cite[Theorem 1.3]{W:15sae5}.

We now establish our iterative complifcation argument.
\begin{lem}\label{comp}
    Suppose that $D_n^{(3)}$ is an auxiliary matrix of type $(n,r,0)_{r,l}$, and $D_n^{(2)}$ is of format $l \times (2\rho_n+l)$. Then there exists an auxiliary matrix $D_{2n}^{(3)}$ of type $(2n,r,0)_{r,l}$ and a matrix $D_{2n}^{(2)}$ of format $l \times (4\rho_n+l)$ such that
    \begin{align*}
        J_n(P) \ll (P^{(31/6) l + \eps})^{1/2} J_{2n}(P)^{1/2}.
    \end{align*}
\end{lem}
\begin{pf}
    As a consequence of the definition of auxiliarity the exponential sums $f(\bg_i)$ with $1 \le i \le \rho$ depend only on $\alpha_{3,1}, \ldots, \alpha_{3,\rho}$, and $f(\bg_i)$ with $\rho+1 \le i \le \rho+l$ depend only on $\ba_{\rho+1}, \ldots, \ba_{\rho+l}$.
    Write now
    \begin{align*}
        V_n(P)= \begin{cases}
        \displaystyle{\sup_{\widetilde{\ba} \in \T^{r-l}}\oint |F_{r-l+1}^{r}(\bg)|^2  |F_{r+1}^{r+l}(\bg)|^8   \D \ba^\dagger} & \text{ for }n=1,\\ \displaystyle{\sup_{\widetilde{\ba}\in \T^{\rho}} \oint |F_{\rho+1}^{\rho+l}(\bg)|^2  |F_{2\rho+1}^{2\rho+l}(\bg)|^{8}   \D \ba^{\dagger}} & \text{ for }n\ge 2,
        \end{cases}
    \end{align*}
    and
    \begin{align*}
        W(P;\ba^{\dagger})&= \oint   |F_{1}^{\rho}(\bg)|^2 |F_{\rho+l+1}^{\rho+2l}(\bg)|^8 |F_{\rho+2l+1}^{2\rho+l}(\bg)|^4    \D \widetilde{\ba},
    \end{align*}
    then by Schwarz' inequality one has
    \begin{align}\label{J-cs}
        J_n(P) \ll V_n(P)^{1/2} \left(\oint |F_{\rho+1}^{\rho+l}(\ba^{\dagger})|^{2} W(P;\ba^{\dagger} )^2 \D  \ba^{\dagger}\right)^{1/2}.
    \end{align}
    We first consider the integral $V_n(P)$. It follows from the triangle inequality that the supremum in the expression for $V_n(P)$ is assumed at $\widetilde{\ba} = \bm 0$; we may therefore neglect all but the lowest $l$ rows in the coefficient matrices $D_n^{(2)}$ and $D_n^{(3)}$. Since $D_n^{(2)}$ is highly non-singular and $D_n^{(3)}$ is auxiliary of type $(n,r,0)_{r,l}$, the submatrices of $D_n^{(2)}$ and $D_n^{(3)}$ given by the last $l$ rows and the columns $r-l+1, \ldots, r+l$ for $n=1$ and $\rho+1, \ldots, \rho+l, 2\rho+1, \ldots, 2\rho+l$ for $n \ge 2$ are still highly non-singular. We may thus apply \eqref{trivineq} and perform a non-singular change of variables, after which an application of \eqref{mv23} leads to the bound
    \begin{align*}
        V_n(P)&\ll \oint \prod_{i=1}^{l} |f(\ba_{\rho+i})|^{10} \ba^{\dagger}\ll \left( \oint |f(\alpha, \beta)|^{10} \D \alpha \D \beta \right)^l \ll P^{(31/6)l+\eps}.
    \end{align*}

    Meanwhile, expanding the square shows that
    \begin{align*}
        W(P;\ba^{\dagger} )^2 = \oint |F_1^\rho(\widehat \bg)|^2|F_{\rho+l+1}^{2\rho+l}(\widehat \bg)|^2 |F_{2\rho+l+1}^{2\rho+2l}(\widehat \bg)|^8 |F_{2\rho+2l+1}^{4\rho}(\widehat \bg)|^4  |F_{4\rho+1}^{4\rho+l}(\widehat \bg)|^8 \D \widetilde{\ba} \D \widetilde{\ba}',
    \end{align*}
    where $\widetilde{\ba}' = (\alpha'_{3,\rho}, \ldots, \alpha'_{3,1})$ and
    \begin{align*}
        \widehat{\bg}_{i}(\ba) =
        \begin{cases}
            \bg_i(\widetilde{\ba} ,\ba^\dagger) & \text{ if } 1 \le i \le \rho+l, \\
            \bg_{2\rho+l+1-i}( \ba^{\dagger},\widetilde{\ba}') & \text{ if } \rho+l+1 \le i \le 2\rho+l,\\
            \bg_{i-\rho}(\widetilde{\ba} ,\ba^\dagger) & \text{ if } 2\rho+l+1 \le i \le 3\rho+l, \\
            \bg_{5\rho+2l+1-i}( \ba^{\dagger},\widetilde{\ba}') & \text{ if } 3\rho+l+1 \le i \le 4\rho+l.
        \end{cases}
    \end{align*}
    It follows that
    \begin{align*}
        &\oint |F_{\rho+1}^{\rho+l}(\ba^{\dagger})|^{2} W(P;\ba^{\dagger} )^2 \D  \ba^{\dagger} \\
        &= \oint |F_1^{2\rho+l}(\widehat \bg)|^2 |F_{2\rho+l+1}^{2\rho+2l}(\widehat \bg)|^8 |F_{2\rho+2l+1}^{4\rho}(\widehat \bg)|^4  |F_{4\rho+1}^{4\rho+l}(\widehat \bg)|^8 \D \widetilde{\ba}  \D\ba^{\dagger} \D \widetilde{\ba}'.
    \end{align*}

    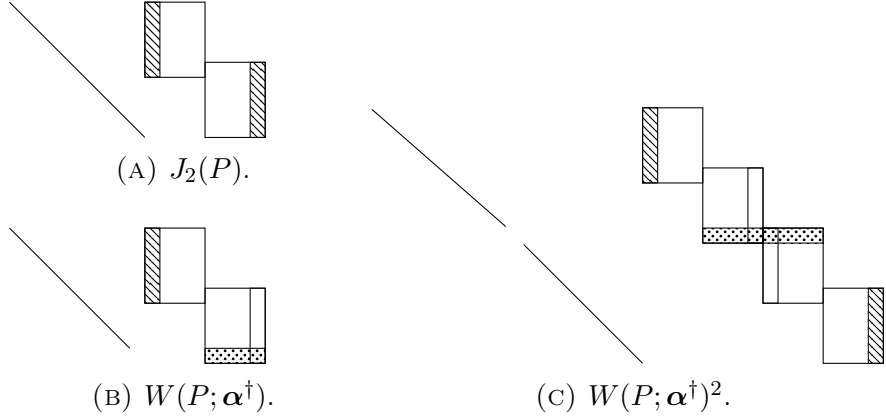
\begin{figure}
     \centering
    \begin{minipage}[b]{0.3\textwidth}
        \begin{subfigure}[b]{\linewidth}
            \begin{tikzpicture}
                \draw (0,1)-- (1.8, -0.8);
                \draw (1.8,0) rectangle (2.6,1);
                \draw [pattern=north west lines] (1.8,0) rectangle (2,1);
                \draw (2.6,0.2) rectangle (3.4,-0.8);
                \draw [pattern=north west lines] (3.2,0.2) rectangle (3.4,-0.8);
            \end{tikzpicture}
            \caption{$J_2(P)$.}
        \end{subfigure}\\[\baselineskip]
        \begin{subfigure}[b]{\linewidth}
            \begin{tikzpicture}
                \draw (0,1) -- (1.6, -0.6);
                \draw (1.8,0) rectangle (2.6,1);
                \draw [pattern=north west lines] (1.8,0) rectangle (2,1);
                \draw (2.6,0.2) rectangle (3.4,-0.8);
                \draw (3.2,0.2) rectangle (3.4,-0.8);
                \draw [pattern=crosshatch dots](2.6,-0.6) rectangle (3.4,-0.8);
            \end{tikzpicture}
            \caption{$W(P;\ba^{\dagger})$.}
        \end{subfigure}
    \end{minipage}
    \begin{subfigure}[b]{0.45\textwidth}
    \begin{tikzpicture}
        \draw (0,1.78) -- (1.78, 0.22);
        \draw (2.02,-0.02) -- (3.6, -1.6);
        \draw (3.6,1.8) rectangle (4.4,0.8);
        \draw (4.4,1) rectangle (5.2,0);
        \draw (5.2,0.2) rectangle (6,-0.8);
        \draw (6,-0.6) rectangle (6.8,-1.6);
        \draw [pattern=north west lines] (3.6,1.8) rectangle (3.8,0.8);
        \draw [pattern=crosshatch dots](4.4,0.2) rectangle (5.2,0);
        \draw (5,1) rectangle (5.2,0);
        \draw (5.2,0.2) rectangle (5.4,-0.8);
        \draw [pattern=crosshatch dots](5.2,0.2) rectangle (6,0);
        \draw [pattern=north west lines] (6.6,-0.6) rectangle (6.8,-1.6);
    \end{tikzpicture}
    \caption{$W(P;\ba^{\dagger})^2$.}
    \end{subfigure}
    \caption{\footnotesize{Schematic representation of the inductive step in the generic case. The matrices correspond to the cubic subsystem. The hatched parts denote to exponential sums to the eighth power, and the dotted rows designate the integrating variables $\ba^{\dagger}$ that do not occur in the integral. The matrix corresponding to the integral $W(P;\ba^{\dagger})$ is obtained from $J_2(P)$ by deleting the last $l$ diagonal elements and reducing the power from 8 to 4 in the last $l$ columns. Squaring it amounts to flipping the matrix (minus the lowest $l$ rows) upside down, but keeping the weights encoded in the dotted part intact. The mean value $J_4(P)$ is now obtained by re-introducing the exponential sums $|F_{\rho+1}^{\rho+l}(\ba^{\dagger})|^{2}$ corresponding to the missing part of the diagonal component and integrating over the corresponding rows.}}
    \end{figure}
    The matrices $\widehat D^{(2)} = (\widehat d^{\;(2)}_{i,j})$ and $\widehat D^{(3)} = (\widehat d^{\;(3)}_{i,j})$ associated to $\widehat \bg_1, \ldots, \widehat \bg_{4\rho+l}$ are of respective formats $l \times (4\rho_n+l)$ and $(2\rho_n+l) \times (4\rho_n+l)$, and the latter is auxiliary of type $(2n, r, 0)_{r,l}$, so the last integral is just $J_{2n}(P)$. This completes the proof of the lemma.
\end{pf}

We can now proceed to prove the mean value estimate that is central to our methods. Lemma~\ref{comp} provides us with the iterating step, which allows us to show that the mean value $J_1(P)$ is subject to nearly square root cancellation.

\begin{thm}\label{K-thm}
    Suppose that the matrices $D^{(3)}$ and $D^{(2)}$ are both highly non-singular and of respective formats $r \times (6r +4l)$ and $l \times (6r +4l)$. We have
    \begin{align*}
        \oint \prod_{i=1}^{6r+4l} |f(\bg_i)| \D \ba \ll P^{3r + \frac{13}{6} l + \eps}.
    \end{align*}
\end{thm}

\begin{pf}
    It follows from \eqref{trivineq} and relabelling that
    \begin{align*}
        \oint \prod_{i=1}^{6r+4l} |f(\bg_i)| \D \ba \ll \oint \prod_{i=1}^r |f(\bg_i)|^2 \prod_{i=r+1}^{r+l} |f(\bg_i)|^{12} \prod_{i=r+l+1}^{2r-l} |f(\bg_i)|^4   \D \ba.
    \end{align*}
    The coefficient matrices of the diophantine system associated to the integral on the right hand side are still highly non-singular. Hence by taking elementary row operations and invoking Lemma~\ref{hns-prop} (ii), we see that the number of solutions to this system is given by a mean value of the shape $J_1(P)$ for suitable matrices $D_1^{(2)}$ and $D_1^{(3)}$, where $D_1^{(2)}$ is highly non-singular of format $l \times (2r-l)$ and $D_1^{(3)}$ is auxiliary of type $(1,r,0)_{r,l}$.
    We may thus deploy Lemma~\ref{comp} which, after $m$ iterations, yields the bound
    \begin{align*}
        J_1(P) \ll (P^{(31/6)l + \eps})^{1-2^{-m}} J_{2^m}(P)^{2^{-m}}
    \end{align*}
    for suitable matrices $D_{2^m}^{(2)}$ and $D_{2^m}^{(3)}$, where $D_{2^m}^{(3)}$ is auxiliary of type $(2^{m}, r,0)_{r,l}$. When $m \ge 1$, it follows from discarding the quadratic equations and estimating the terms $|F_{\rho+l+1}^{\rho+2l}|^4|F_{2\rho+1}^{2\rho+l}|^4$ trivially that
    \begin{align*}
        J_{2^m}(P) \ll P^{8l} I(P, D_{2^m}^{(3)}).
    \end{align*}
    Combining these estimates and inserting Proposition \ref{silly} yields the bound
    \begin{align*}
        \oint \prod_{i=1}^{6r+4l} |f(\bg_i)| \D \ba  & \ll (P^{\frac{31}{6}l})^{1-2^{-m}} (P^{8l} P^{3 (2^{m}-1)(r-l) + 3r -2l +\eps})^{2^{-m}} \\
        &\ll P^{3r + \frac{13}{6} l +  2^{-m}  \frac{23}{6}l+\eps}.
    \end{align*}
    The result now follows on letting $m$ tend to infinity.
\end{pf}

\section{The Hardy-Littlewood Method}\label{CM}

We now have the means at hand to complete the proof of Theorem~\ref{thm}. The treatment here is a straightforward adaptation of the arguments of \cite[\S 4 and \S 6]{BP:16}. We take $r=r_3$ and $l=r_2$  and write $w=r_3+r_2$. Set
\begin{align}\label{s-bd}
    s \ge 6r_3 + \lfloor (14/3) r_2 \rfloor +1,
\end{align}
and make the change of variables
\begin{align*}
    \gamma_{k,j}(\ba) = \sum_{i=1}^{r_k}c_{i,j}^{(k)}\alpha_{k,i} \quad (1 \le j \le s, k \in \{2,3\}).
\end{align*}
Set $\bg_j = (\gamma_{3,j}, \gamma_{2,j})$, and write $f(\bg_j) = f_j(\ba)$. When $\F B$ is measurable, let
\begin{align*}
    N(P; \F B) = \int_{\F B} \prod_{i=1}^s f_i(\ba) \D \ba.
\end{align*}
We define two sets of major arcs. Let
\begin{align*}
    \F M(q,\B a) = \big\{\ba \in \T^w: |q\alpha_{k,i}-a_{k,i}| \le P^{3/4-k} \qquad (1\le i \le r_k,  k \in \{2,3\})\big\},
\end{align*}
and write $\F M$ for the union of all $\F M(q,\B a)$ with $1 \le \B a \le q$, $(q,\B a)=1$, and $1 \le q \leqslant P^{3/4}$.  We then write $\F m = \T^w \setminus \F M$ for the minor arcs. Set now $X=P^{1/(6w)}$ and define further
\begin{align*}
    \F N = \bigcup_{q=1}^X \bigcup_{\substack{\B a =1  \\ (q, \B a)=1}}^q \F N(q,\B a),
\end{align*}
where $\F N(q,\B a)$ is given by the set of all $\ba  \in \T^w$ satisfying
\begin{align*}
     |\alpha_{k,i}-q^{-1}a_{k,i}| \le XP^{-k} \qquad (1 \le i \le r_k, k \in \{2,3\}),
\end{align*}
and $\F n =\T^w \setminus \F N$. It then follows that
\begin{align*}
    N(P) = N(P;\F N) + O(N(P;\F m )) + O(N(P;\F M\setminus \F N)).
\end{align*}

We first consider the contribution from the minor arcs. Set $\sigma = s-6r_3-4r_2$, then a straightforward modification of the arguments of \cite[\S 6]{BP:16} shows that
\begin{align*}
    N(P;\F m) &\ll P^{(3/4)\sigma + \eps} \oint \prod_{i=1}^{6r_3 + 4r_2} |f(\bg_i)| \D \ba \ll P^{(3/4)\sigma + \eps} P^{3r_3 + \frac{13}{6} r_2 + \eps},
\end{align*}
where the last step uses Theorem \ref{K-thm}.
An easy computation confirms that the exponent is smaller than $s - 2r_2 - 3r_3$ whenever $s>6r_3+(14/3) r_2$ and $\eps$ has been chosen small enough. Similarly, it follows from Lemma~6.2 of \cite{BP:16} that whenever $s$ satisfies \eqref{s-bd}, then one has
\begin{align*}
    N(P;\F M\setminus \F N) \ll P^{s-2r_2-3r_3}X^{-1/(6r_3)}.
\end{align*}
Altogether, we obtain
\begin{align}\label{minbd}
    N(P) = N(P;\F N) + O(P^{s-2r_2-3r_3-\delta})
\end{align}
for some small $\delta>0$. This completes the analysis of the minor arcs $\F n$.

Finally, the treatment of the major arcs is precisely as in \cite[\S 4]{BP:16}. Write $\alpha_{k,i} = q^{-1}a_{k,i} + \beta_{k,i}$, and let
\begin{align*}
    S(q,\B a) = \sum_{x=1}^q e\left((a_3 x^3 + a_2 x^2)/q\right)
\end{align*}
and
\begin{align*}
    v(\bb,P) = \int_{-P}^P e(\beta_3 z^3 + \beta_2 z^2) \D z.
\end{align*}
We make the change of variables
\begin{align*}
    \Lambda_{k,j} = \sum_{i=1}^{r_k}c^{(k)}_{i,j} a_{k,i} \quad \text{and} \quad  \delta_{k,j}= \sum_{i=1}^{r_{k}} c^{(k)}_{i,j} \beta_{k,i} \qquad (1 \le j \le s, k \in \{2,3\})
\end{align*}
and write $\bm \Lambda_{j}=(\Lambda_{3,j},\Lambda_{2,j})$ and $\bm \delta_{j}=(\delta_{3,j},\delta_{2,j})$,
so that $\bd_{j}=\bg_{j}-\B \Lambda_{j}/q$ for $ 1 \le j \le s$. Set $S_j(q,\B a)=S(q,\bm{\Lambda}_j)$ and $v_j(\bb, P)=v(\bd_j, P)$, and for arbitrary $Y$ define the truncated singular series
\begin{align*}
    \F S(Y) = \sum_{q \le Y} \sum_{\substack{1 \leqslant \B a \leqslant q \\ (q,\B a)=1}} \prod_{j=1}^s q^{-1} S_j(q,\B a)
\end{align*}
and singular integral
\begin{align*}
    \F J(Y) = \int_{[-YP^{-2}, YP^{-2}]^{r_2}} \int_{[-YP^{-3}, YP^{-3}]^{r_3}} \prod_{j=1}^s v_j(\bb,P) \D \bb.
\end{align*}
Then it follows from the arguments leading to equation (4.5) of \cite{BP:16} that
\begin{align}\label{majorasymp}
\int_{\F N} f(\bg_1) \cdot \ldots \cdot f(\bg_s) \D \ba = \F S(X) \F J(X) + O(P^{s-2r_2-3r_3-\delta})
\end{align}
for some $\delta > 0$.

One can complete the singular series and singular integral as usual by taking $\F S = \lim_{Y \to \infty} \F S(Y)$ and $\F J=\lim_{Y \to \infty} \F J(Y)$, where $P$ is held fixed. Then it is shown in \cite[Lemmata 4.1 and 4.2]{BP:16} that $\F S - \F S(Y) \ll Y^{-\delta}$ and $\F J - \F J(Y) \ll P^{s-2r_2-3r_3}Y^{-\delta}$ for some $\delta>0$ whenever $s$ satisfies \eqref{s-bd}.
A coordinate transform now shows that $\F J = P^{s-2r_2-3r_3} \chi_{\infty}$ with
\begin{align*}
    \chi_\infty = \int_{\R^r} \int_{[-1,1]^s} e\Big(\sum_{i=1}^{r_3} \beta_{3,i} \Theta_{3,i}(\bm{\zeta}) + \sum_{i=1}^{r_2} \beta_{2,i} \Theta_{2,i}(\bm{\zeta})\Big) \D \bm{\zeta}  \D \bb,
\end{align*}
where $\Theta_{k,i}(\B x) = c_{i,1}^{(k)}x_1^{k} + \ldots + c_{i,s}^{(k)}x_s^{k}$, and our above arguments imply that $\chi_\infty$ is a finite constant. Furthermore, the argument of \cite[Lemma~7.4]{P:02pae} is easily adapted to prove that this constant is positive whenever the system \eqref{sys} possesses a non-singular real solution in the unit hypercube.
Also, a standard argument yields
\begin{align*}
    \chi_p = \sum_{i=0}^\infty A(p^i) = \lim_{i \to \infty} p^{-i(s-r)} M(p^i),
\end{align*}
where $M(p^i)$ denotes the number of solutions of the congruences to the modulus $p^i$ that correspond to the equations \eqref{sys}. It follows from standard arguments that $\chi_p = 1+O(p^{-1}) \ge \tfrac{1}{2}$ for $p$ sufficiently large, and for small primes one uses Hensel's lemma to deduce that $\chi_p>0$ if the system \eqref{sys} possesses a non-singular $p$-adic solution.  The proof of Theorem~\ref{thm} is now complete on recalling \eqref{minbd} and \eqref{majorasymp}, and the constant  is given by $ c = \chi_\infty \prod_{p}\chi_p$.

\bibliographystyle{amsplain}
\bibliography{fullrefs}
\end{document}